\documentclass[12pt,reqno]{amsart}
\usepackage{amsmath}
\usepackage{amssymb, latexsym, amsfonts, amscd, amsthm, mathrsfs, enumerate, esint}
\usepackage[usenames,dvipsnames]{color}
\usepackage[all]{xy}
\usepackage{graphicx}
\usepackage{epsfig}
\usepackage{hyperref}

\definecolor{purple}{rgb}{0.9,0,0.8}

\definecolor{gray}{rgb}{0.7,0.7,0.7}

\newcommand{\abbr}[1]{{\sc\lowercase{#1}}}

\newtheorem{thm}{Theorem}[section]

\newtheorem{lem}[thm]{Lemma}
\newtheorem{ppn}[thm]{Proposition}

\theoremstyle{definition}
\newtheorem{defn}[thm]{Definition}
\newtheorem{ex}[thm]{Example}
\newtheorem{remark}[thm]{Remark}
\newtheorem{assumption}[thm]{Asumption}

\newcommand{\beq}{\begin{equation}}
\newcommand{\eeq}{\end{equation}}

\newcommand{\ep}{\epsilon}

\newcommand{\B}{\mathbb{B}}

\newcommand{\D}{\mathbb{D}}
\newcommand{\bD}{\mathbb{D}}
\newcommand{\E}{\mathbb{E}}
\newcommand{\bE}{\mathbb{E}}

\newcommand{\bG}{\mathbb{G}}
\newcommand{\G}{\mathbb{G}}

\newcommand{\bH}{\mathbb{H}}

\newcommand{\K}{\mathbb{K}}

\newcommand{\N}{\mathbb{N}}

	\renewcommand{\P}{\mathbb{P}}	
\newcommand{\bP}{\mathbb{P}}

\newcommand{\Z}{\mathbb{Z}}
\newcommand{\bZ}{\mathbb{Z}}

\newcommand{\cF}{\mathcal{F}}

\newcommand{\wt}{\widetilde}
\newcommand{\wh}{\widehat}

\DeclareMathOperator{\argmax}{arg\,max}

\renewcommand{\emptyset}{\varnothing}

\begin{document}
\title{On random walk on growing graphs}
\author{Ruojun Huang}
\address{Department of Statistics, Stanford University. Sequoia Hall, 390 Serra Mall, Stanford, CA 94305, USA.}
\thanks{This research was supported in part by NSF grant DMS-1106627.}
\date{\today}
\maketitle
\begin{abstract}
Random walk on changing graphs is considered. For sequences of finite graphs increasing monotonically towards a limiting infinite graph, we establish transition probability upper bounds. It yields sufficient transience criteria for simple random walk on slowly growing graphs, upon knowing the volume and Cheeger constant of each graph.
For much more specialized cases, we establish matching lower bounds, and deduce sufficient (weak) recurrence criteria. We also address recurrence directly in relation to a universality conjecture of \cite{DHS}.
We answer a related question of \cite[Problem 1.8]{SZ} about ``inhomogeneous merging" in the negative.
\end{abstract}

\section{Introduction}\label{sec:intro}
This work pursues an interest in behaviors of random walk on time-dependent graphs. The time evolution of the graph is assumed to be independent of the walk, to distinguish from interacting-type models, resulting in its random walk forming a time-inhomogeneous Markov chain. There is a sizeable literature on random walk in dynamic random environment, which is partly motivated by application to random walk on a field of moving particles in equilibrium (see e.g. \cite{HHSST, MO, ACDS} and references therein).
Our framework differs, in that we assume no other abstract condition than that the graph is a ``monotone graph", and hence far from equilibrium. A quintessential example is random walk on growing-in-time $d$-dimensional domains, analysed in \cite{DHS}. Motivated by determining recurrence versus transience of simple random walk on independently growing Internal Diffusion Limited Aggregation (\abbr{IDLA} \cite{LBG}) clusters on $\bZ^d$, it is proved that for any increasing sets $\D_t\uparrow\Z^d$, $d\ge 3$, having the rough shape of a ball $\B_{f(t)}\subseteq\D_t\subseteq\B_{Cf(t)}$, for some $f(t)\uparrow\infty$ and finite constant $C$, whenever $\int_1^\infty f(t)^{-d}dt<\infty$, the walk $\{X_t\}$ which takes steps in $\{\D_t\}$ almost surely visits every vertex of $\Z^d$ finitely often; and under additional technical conditions, the converse is also true. In particular, there is a recurrent phase when the domain grows sufficiently slowly, despite that the limiting graph is transient.

The method used in \cite{DHS} is specific to ball-like sets in $d$-dimensional lattice, but the question of obtaining sufficient criteria to determine transience or recurrence extends to general increasing graphs.
Recent works conjecture and partly show that some universality applies for monotonically and independently time-varying graphs. 
Our model as stated, is in fact a degenerate case (for $\Pi_t\in\{0,1\}^{E}$) of a more general framework introduced by \cite{ABGK}, of discrete time random walk $\{X_t\}$ on graphs $\G=(V,E)$ endowed with time-dependent (symmetric non-negative) edge conductances $\{\Pi_t\}$ - more details see section \ref{framework}. Restricting to the independent setting, they propose a universality conjecture, which states that provided the conductances $\Pi_t\in[0,\infty)^{E}$ are edge-wise non-decreasing in $t$ and both the starting and ending conductances $\Pi_0,\Pi_\infty$ correspond to recurrent (resp. transient) graphs, then the dynamic model is also recurrent (resp. transient). Without monotonicity, however, this is known to be false. Extending the use of potential theory, the conjecture is verified in \cite{ABGK} in case $\G$ is any tree. Later, \cite{DHMP} further verifies the transient case when uniform in $t$ isoperimetric inequality of order $d>2$ holds for $(\G,\Pi_t)$ - for example $\G=\Z^d, d>2$ and $\Pi_t$ bounded uniformly up and below. The method used, of establishing heat kernel estimates for the time-inhomogeneous random walk via {\emph{evolving sets}}, is further developed in the present work to apply to more degenerate cases (cf. the relevant \cite[Problem 1.17]{DHMP}). 
\footnote{See \cite{DHZ} for a further development in the uniformly elliptic conductance case, where Gaussian-type two-sided estimates in the sprirt of \cite{Del, HS} are established via analytic means. The focus of the present paper is considerably different from \cite{DHZ}, in particular the graph structure is changing rather than just the conductances and we prefer to take a probabilistic route.}
Evolving random sets was introduced in \cite{MP1,MP} to give improved bounds on mixing times of  finite Markov chains, as well as a probabilistic derivation of on-diagonal heat kernel upper bounds on infinite graphs from graph isoperimetric properties. Equipped with a time-dependent version of evolving sets, we use the idea of (imperfect) mixing to give heat kernel upper bounds for finite graphs growing slowly towards a limiting (infinite) graph. See Theorem \ref{thm-main} and Proposition \ref{ppn:tra}. Of course, for inhomogeneous Markov chains the notion of mixing (or any quantitative statements about their ergodic properties) is delicate, and without imposing strong structural assumptions it is generally anomalous.

Indeed, one other purpose of this work is to answer a related question posed in \cite[Problem 1.8]{SZ} about {\emph{merging}} of inhomogeneous finite Markov chains. Studied in a sequence of works \cite{SZ1,SZ} is the following problem: given time-dependent Markov transition kernels $K^{(t)}$ (say, on a graph), each of which is reversible with respect to some {\it probability} measure $\mu^{(t)}$, under what conditions will the Markov chain $\{X_t\}$ that uses $\{K^{(t)}\}$ as its transitions forget its initial condition? The latter property is called {\emph{merging}}, and the quantitative bounds on the time to achieve such, relative to the size or complexity of the system, is called merging time. It can be viewed as analogue of mixing time for homogeneous finite Markov chains (\cite{LPW}). Techniques such as Nash and log-Sobolev inequalities are developed for this purpose, under an overarching assumption called {\emph{$c$-stability}} \cite[section 1.4]{SZ}, which however is hard to verify. Hence the following question is left open: for birth-death processes on intergers $V_N=\{0,1,...,N\}$, with each kernel $K^{(t)}(x,y)\in[1/4,3/4]$ whenever $|x-y|\le 1$, and its reversible measure $(N+1)\mu^{(t)}(x)\in[1/4,4]$ for all $x\in V_N$, whether the total variation $\delta$-merging time
\begin{align}\label{tv-merge}
T_{TV}(\delta):=\inf\Big\{t\ge0:\max_{x,y\in V_N}||K_{0,t}(x,\cdot)-K_{0,t}(y,\cdot)||_{TV}<\delta\Big\}
\end{align}
where $K_{0,t}:=K^{(0)}K^{(1)}...K^{(t-1)}$,
or relative-sup $\delta$-merging time 
\begin{align}\label{sup-merge}
T_\infty(\delta):=\inf\Big\{t\ge0:\max_{x,y,z\in V_N}\left|\frac{K_{0,t}(x,z)}{K_{0,t}(y,z)}-1\right|<\delta\Big\}
\end{align}
must be at most $CN^2(1+\log(\delta^{-1}\vee1))$ for some $C$ universal. We provide a negative answer.
\begin{ppn}\label{ppn:no-merging}
For any $\ep>0$, there exists birth-death process on $V_N=\{0,1,..,N\}$, each step $t$ using kernel $K^{(t)}$ satisfying $K^{(t)}(x,y)\in[1/3-\ep,1/3+\ep]$ whenever $|x-y|\le 1$, except for $K^{(t)}(0,0), K^{(t)}(N,N)\in[2/3-\ep,2/3+\ep]$; and each $K_t$ having reversible probability measure $\mu^{(t)}$ satisfying $(N+1)\mu^{(t)}(x)\in [1-\ep,1+\ep]$ for all $x\in V_N$, such that both the total variation and relative-sup $\frac{1}{2}$-merging times are at least $\alpha e^{\alpha N}$ for some $\alpha=\alpha(\ep)>0$.
\end{ppn}

\subsection{Framework and main results}\label{framework}
Let $\bG_\infty=(V,E)$ be an infinite locally finite connected graph with vertex set $V$ and edge set $E$, allowing for multiple edges and self-loops. Consider an increasing sequence of subgraphs $\bG_t=(V,E_t)$ of $\bG_\infty$, having the same set $V$ of vertices and increasing sets of edges $E_t\subseteq E_{t+1}...\subseteq E$. Write
\begin{align*}
V_t:=\{x\in V:\, \exists y\in V, \, y\neq x \text{ such that }(x,y)\in E_t\}
\end{align*}
for the set of non-isolated vertices at time $t$. We will identify $\G_t$ with $(V_t,E_t)$ without loss of generality, rendering every $\G_t$ connected. Throughout we use $|\cdot|$ to denote cardinality. For each $t\in\N$, let
\begin{align}\label{def:degree}
\pi^{(t)}(x,y):=\left|\{e\in E_t:\, e=(x,y)\}\right|
\end{align}
be the number of multiple edges between vertices $x,y\in V_t$, with $\pi^{(t)}(x,x)$ counting self-loops at $x$. \footnote{Setting $\pi^{(t)}(x,y)\equiv 0$ for $(x,y)\in E\backslash E_t$, we may also identify $\G_t$ with $(V,E,\Pi_t)$ as we did when discussing general conductance models on pages 1-2, although we will not be using this notation in the sequel.}
Let for any $x\in V_t$, $A, B\subseteq V_t$
\begin{align}\label{def:pi}
\pi^{(t)}(x):=\sum_{y\in V_t}\pi^{(t)}(x,y), \quad \pi^{(t)}(A,B):=\sum_{x\in A,\, y\in B}\pi^{(t)}(x,y), \quad  \pi^{(t)}(A):=\sum_{x\in A}\pi^{(t)}(x).
\end{align}
That $\{\bG_t\}$ is an increasing sequence of subgraphs means that both $t\mapsto V_t$ and $t\mapsto\pi^{(t)}(x)$ are non-decreasing.
We assume throughout that all $\bG_t$, $t<\infty$, are finite graphs, i.e. $|V_t|<\infty$, for which we define the following quantities. Let the volume of $\bG_t$ be denoted by
\begin{align}\label{def:vol}
v(t):=\pi^{(t)}(V_t),
\end{align}
and its isoperimetric function (with convention $\inf\emptyset=\infty$)
\begin{align}\label{def:isop}
\phi_t(r):=\inf_{\substack{A\subseteq V_t:\,\\
 \pi^{(t)}(A)\le \frac{v(t)}{2}\wedge r}}
\Big\{\frac{\pi^{(t)}(A,A^c)}{\pi^{(t)}(A)}\Big\}, \quad r\ge 0
\end{align}
whereby the Cheeger constant (aka bottleneck ratio) is thus
\begin{align}\label{def:cheeger}
\Phi_t:=\phi_t\big(\frac{v(t)}{2}\big).
\end{align}
Consider simple random walk (\abbr{SRW}) $\{X_t\}_{t\in\N}$ with $X_0=x_0\in V_0$ on the increasing sequence of subgraphs $\{\bG_t\}$, having time-inhomogenous Markovian transition probability at each step $t$:
\begin{align}\label{eq:kernel}
P(t,x;t+1,y):=\bP(X_{t+1}=y|X_t=x)=\frac{\pi^{(t)}(x,y)}{\pi^{(t)}(x)}, \quad x,y\in V_t.
\end{align}
We now state a transition probability upper bound for $\{X_t\}$ on $\{\bG_t\}$. 

\begin{thm}\label{thm-main}
Assume $\{X_t\}$ is uniformly $\gamma$-lazy, i.e. $P(t,x;t+1,x)\ge \gamma$ for some $\gamma\in(0,1/2]$ and all $x\in V_t$. Then for any $\alpha\in(0,1)$ and $c_+=\frac{2\alpha(1-\alpha)\gamma^2}{(1-\gamma)^2}$, all $x_0\in V_0$, $y\in V_t$ and $t\ge2$,
\begin{align}\label{first-bd}
P(0,x_0;t,y)\le \min_{1\le s\le t-1}\left\{\frac{2\pi^{(t)}(y)}{v(s)}+\pi^{(t)}(y)^{1-\alpha}L^{(s)}_t\right\},
\end{align}
where $L^{(s)}_t$ is iteratively determined from $L^{(s)}_s=\pi^{(0)}(x_0)^{\alpha-1}$ and
\begin{align}\label{eq:iter}
L^{(s)}_{u+1}:=\argmax\Big\{\ell:\,\int_{\ell/2}^{L^{(s)}_u/2}\frac{dz}{c_+z\phi^2_u\big(z^{\frac{1}{\alpha-1}}\big)}\ge1\Big\}, \quad s\le u\le t-1.
\end{align}
In particular, assuming $\pi^{(t)}(y)$ are uniformly bounded by constant $\Delta$, then for some finite $C=C(\gamma,\Delta)$, $c_\star=\frac{\gamma^2}{2(1-\gamma)^2}$, and all $x_0\in V_0$, $y\in V_t$, $t\ge2$,
\begin{align}\label{second-bd}
P(0,x_0;t,y)\le C\Big(\frac{1}{v(\lfloor t/2\rfloor)}+e^{-c_\star\sum_{u=\lfloor t/2\rfloor}^{t-1}\Phi^2_u}\Big).
\end{align}
\end{thm}

\begin{remark}\label{rmk:why}
The second bound (\ref{second-bd}) is effective only when $\{\bG_t\}$ have slow growth or good connectivity, so that there is some mixing phenomenon happening. In such cases one expects the transition probability to not exceed constant over the volume. The first bound (\ref{first-bd}) works for more general cases, but explicit expressions are harder to obtain, due to the dual effects of mixing to uniform in finite graphs $\bG_t$ and behaving as \abbr{SRW} in the limiting infinite graph $\bG_\infty$, depending on slow or fast growth in different regions.
\end{remark}

Our definition of recurrence and transience is special, due to a lack of zero-one law.

\begin{defn}\label{def:rec}
The stochastic process $\{X_t\}$ with $X_0=x_0\in V_0$ is called (strong) transient, if the expected number of returns to $x_0$ (hence also to every other point) is finite, i.e. $\E_{x_0}[N_0]<\infty$ for $N_0:=\sum_{t=0}^\infty1_{\{X_t=x_0\}}$. Conversely, it is called (weak) recurrent, if $\E_{x_0} [N_0]=\infty$.
\end{defn}

\begin{remark}
When 
\begin{align}\label{pal-zyg}
\limsup_{k\to\infty}\Big\{\frac{\E_{x_0}[N_0(k)]^2}{\E_{x_0}[N_0(k)^2]}\Big\}>0,
\end{align}
for $N_0(k):=\sum_{t=0}^k1_{\{X_t=x_0\}}$, one can further deduce $\P_{x_0}(N_0=\infty)>0$ from $\E_{x_0}[N_0]=\infty$ by applying Paley-Zygmund inequality (cf. \cite[Lemma 2.1]{GP}).
\end{remark}

We deduce from \eqref{second-bd} a sufficient transience criterion.
\begin{ppn}\label{ppn:tra}
Assume $\{X_t\}$ is uniformly lazy and $\{\G_t\}$ have uniformly bounded degrees. Then $$
\sum_{t=1}^\infty v(t)^{-1}<\infty 
$$
is sufficient criterion for $\{X_t\}$ to be transient, if in addition \begin{align}\label{eq:mixing}
\sum_{t=1}^\infty\exp\Big\{-\sum_{u=\lfloor t/2\rfloor}^{t-1}\Phi_u^2\Big\}<\infty. 
\end{align}
\end{ppn}

We give below an example of growing domains in $\bG_\infty=\bZ^d$ to illustrate the use of Theorem \ref{thm-main} and Proposition \ref{ppn:tra}.

\begin{ex}\label{ex:zd}
Consider \abbr{SRW} $\{X_t\}$ on $\bD_t\subseteq\bD_{t+1}...\subseteq\D_\infty=\bZ^d$, $d>2$. Assume each $\bD_t$ has Cheeger constant $\Phi_t\ge c_dv(t)^{-1/d}$, which is the best possible and is satisfied for sufficiently regular sets. Assume $t\mapsto v(t)$ is of polynomial growth, that is, $0<\liminf v(t)/t^\beta\le \limsup v(t)/t^\beta<\infty$ for some $\beta>0$. Then for $\beta<d/2$,
\begin{align*}
\sum_{t=1}^\infty e^{-\sum_{u=\lfloor t/2\rfloor}^{t-1}\Phi^2_u}
=\sum_{t=1}^\infty e^{-c(d)t^{1-2\beta/d}}<\infty.
\end{align*}
Combined with the convergence of $\sum_t v(t)^{-1}$ when $\beta>1$, we conclude that $\beta\in(1,d/2)$ is sufficient for $\{X_t\}$ to be transient.

For faster growth of $\beta\ge d/2$, the second bound (\ref{second-bd}) cannot provide any information since the behavior of $\{X_t\}$ is now closer to the \abbr{SRW} on $\D_\infty=\bZ^d$. One may appeal to the first bound (\ref{first-bd}) instead, but then needs to know the whole isoperimetric function $\phi_t(r)$ of (\ref{def:isop}) rather than just the Cheeger constant $\Phi_t$ of (\ref{def:cheeger}).
For example, if one knew $\phi_t(r)\ge c_d\left(r\wedge \frac{v(t)}{2}\right)^{-1/d}$ for all $t$ and $r$, then (\ref{first-bd})-(\ref{eq:iter}) yield after some algebra, with $s=\lfloor t/2\rfloor$,
\begin{align*}
P(0,x_0;t,y)\le C(d,\alpha,\gamma)\Big(t^{-\frac{d(1-\alpha)}{2}}\vee v(t)^{-(1-\alpha)}\Big).
\end{align*}
Recall $d>2$, and for every $\beta\ge d/2$ we can choose $\alpha\in\left(0,1-\frac{2}{d}\right)$ that makes the \abbr{RHS} integrable, thereby extending the transience conclusion from $\beta\in(1,d/2)$ to all $\beta>1$.
\end{ex}

We give another example of growing subgraphs $\{\G_t\}$ that have fast mixing properties that are expander-like.
\begin{ex}
Consider $\bG_t\subseteq\bG_{t+1}...\subseteq\bG_\infty$ of uniformly bounded degrees, that satisfy $\Phi_t\ge \delta$ for some uniform constant $\delta>0$ and all $t$. 
Then clearly $\sum_{t=1}^\infty v(t)^{-1}<\infty$ is sufficient for $\{X_t\}$ to be transient, as \eqref{eq:mixing} is already satisfied.
\end{ex}

\subsection{Sufficient conditions for recurrence}
For much more specialized cases, we establish matching lower bounds. To this end, we introduce additional definitions and notations.
The inner boundary of a subgraph $\bH$ {\emph{relative to $\G_\infty$}} is denoted by
\begin{align*}
\partial \bH:=\{x\in \bH:\, \exists y\in \G_\infty\backslash \bH \text{ such that }(x,y)\in E\},
\end{align*}
and by $\tau_\bH$ is denoted the first hitting time of $\partial\bH$
\begin{align*}
\tau_\bH:=\{t\ge0:\, X_t\in\partial\bH\}.
\end{align*}

\begin{defn}\label{def:induce}
We say that the edge set of subgraph $\bH\subseteq\G_\infty$ is {\emph{induced from}} $\G_\infty$, if for every $x,y\in\bH$, we have $(x,y)\in E(\bH)$ whenever $(x,y)\in E$. In such case, we abbreviate $v(\bH):=\pi^{(\infty)}(\bH)$.
\end{defn}

Let $d(x,y)$ denote the graph distance between $x,y$ in $\bG_\infty$, and $\B(x,R)$ the closed $\G_\infty$-ball of radius $R$ around $x\in V$.
Below we focus on increasing sequences $\{\bG_t\}$ that satisfy for some $x_0\in V_0$ and positive continuous non-decreasing function $r(t)$,
\begin{align}\label{ball-like} 
t\ge0, \quad \B(x_0,r(t))\subseteq \bG_t, \quad \lim_{t\to\infty}\frac{v(t)-v(\B(x_0,r(t)))}{v(t)}=0.
\end{align}
In words, such $\{\bG_t\}$ have the geometric property that each contains a complete (i.e. with edge set induced) ball of the limiting graph $\bG_\infty$, and the fluctuation at the boundary is of smaller order volume.

Let $p_t^m(x,y)$ denote the (Dirichlet) heat kernel of \abbr{SRW} $\{Y_t\}$ on $\bG_\infty$ that is killed upon first hitting $\partial\B(x_0,m)$, i.e.
\begin{align*}
t\ge0, \quad p_t^m(x,y):=\bP_x\left(Y_t=y,\, t<\tau_{\B(x_0,m)}\right),
\end{align*}
with $p_t(x,y)=\P_x(Y_t=y)$ the un-killed heat kernel.

\begin{defn}\label{def:dhk}
We say that $\bG_\infty$ satisfies the property LLE($\psi$)
{\footnote{Local lower estimate.}} 
with positive continuous non-decreasing function $\psi(\cdot)$, if for any $\delta\in(0,1)$,
there exists constant $c_{\text{HK}}>0$ depending on $\delta$ such that for all $m\ge 1$ and $x,y\in \B(x_0,(1-\delta)m)$,
\begin{align}\label{eq:dhk}
\text{LLE}(\psi):  \quad\quad  p^m_{\psi(m)}(x,y)\ge\frac{c_{\text{HK}}}{v(\B(x_0,m))}.
\end{align}
\end{defn}
\begin{remark}
For example, $\G_\infty=\Z^d$ satisfies property (\ref{eq:dhk}) with $\psi(m)=m^2$.
More generally, this property is satisfied for any infinite graph that admits two-sided heat kernel estimates of the following form: for some $\beta_2\ge\beta_1\ge1$, constants $A_1,...,A_5$, and all $x,y\in V$ with $1\vee d(x,y)\le t$,
\begin{align}
\text{HKE}(\psi) : \quad &\frac{p_t(x,y)}{\pi(y)}\le\frac{A_1}{v(\B(x,\psi^{-1}(t)))}\exp\Big\{-A_2\Big(\frac{\psi(d(x,y))}{t}\Big)^{\frac{1}{\beta_2-1}}\Big\},\label{sub-gauss-1}\\
&\frac{p_t(x,y)+p_{t+1}(x,y)}{\pi(y)}\ge\frac{A_3}{v(\B(x,\psi^{-1}(t)))}\exp\Big\{-A_4\Big(\frac{\psi(d(x,y))}{t}\Big)^{\frac{1}{\beta_2-1}}\Big\},\label{sub-gauss-2}
\end{align}
when $\psi(\cdot)$ satisfies
\begin{align}\label{eq:poly-growth}
\Big(\frac{R}{r}\Big)^{\beta_1}\le\frac{\psi(R)}{\psi(r)}\le A_5\Big(\frac{R}{r}\Big)^{\beta_2}, \quad \forall\, 0< r\le R,
\end{align}
where $\psi^{-1}$ is the inverse of $\psi$.
The case when $\psi(m)=m^\beta$, $\beta\ge 2$, is called sub-Gaussian heat kernel estimates, with $\beta=2$ the Gaussian case. In fact, the properties LLE($\psi$) (\ref{eq:dhk}) and HKE($\psi$) (\ref{sub-gauss-1})-(\ref{sub-gauss-2}) are equivalent under (\ref{eq:poly-growth}). This is well known to experts, see for example \cite[Theorem 3.2]{BGK} in the framework of metric measure spaces equipped with strongly local regular symmetric Dirichlet forms. 
The HKE($\psi$) are in turn equivalent to a combination of volume doubling property (see Definition \ref{def:vd}), Poincar\'e inequality of scale $\psi$ and a family of cut-off Sobolev inequalities of scale $\psi$ (see \cite[Theorem 2.16]{BBK} for definitions of the latter two concepts and other details). Thus many fractal graphs are also included (e.g. pre-Sierpinski gaskets, vicsek sets, and more generally the so-called nested fractals with $\beta=d_w$ the walk dimension of the graph).
\end{remark} 

\begin{defn}\label{def:vd}
We say that a positive non-decreasing function $s\mapsto f(s)$ is doubling, if $f(2s)\le Df(s)$ for some $D<\infty$ and all $s>0$. 
\end{defn}

\begin{ppn}\label{ppn-main}
Assume $\{X_t\}$ is uniformly $\gamma$-lazy and $\pi^{(t)}(x)$ are uniformly bounded by constant $\Delta$. Assume further that $\G_\infty$ satisfies property LLE($\psi$) of Definition \ref{def:dhk} as well as
\begin{align}   \label{cond:fluct}
\lim_{\delta\downarrow 0}\limsup_{m\to\infty}\left[1-\frac{v(\B(x_0, (1-\delta)m))}{v(\B(x_0, m))}\right]=0\, ;
\end{align}
the increasing sequence $\{\bG_t\}$ satisfies the geometry condition (\ref{ball-like}), the regularity condition
\begin{align}\label{eq:regularity}
\liminf_{t\to\infty}\left\{\frac{\sum_{u=\lfloor t/2\rfloor}^{t-1}\Phi^2_u}{\log v(\lfloor t/2\rfloor)}\right\}=:\zeta>0,
\end{align}
and $t\mapsto v(t)$ is a doubling function with constant $D$. Then for some positive $c=c(\gamma, \Delta, c_{\text{HK}}, \\
\zeta, D)$ and $\delta_0=\delta_0(\gamma,\Delta,\zeta, D)\le\frac{1}{2}$, all $t\ge 2$, $y\in V_t$ with $d(x_0,y)\le (1-\delta_0)(r^{-1}+\psi)^{-1}(t)$,
\begin{align}\label{eq:lbd}
P(0,x_0;t,y)\ge\frac{c}{v(t)}.
\end{align}
\end{ppn}

\begin{remark}
Thus, $\sum_{t=1}^\infty v(t)^{-1}=\infty$ is sufficient for $\{X_t\}$ to be (weak) recurrent, in the setting of Prop. \ref{ppn-main}. Combined with Proposition \ref{ppn:tra}, it establishes a sharp phase transition between transience and recurrence.
\end{remark}

\begin{remark}
The condition (\ref{eq:regularity}) is to guarantee that the second term in the upper bound (\ref{second-bd}) is absorbed in the first term. The main restriction of Prop. \ref{ppn-main}, compared to Theorem \ref{thm-main}, is the ball-like assumption (\ref{ball-like}) on the geometry of $\{\G_t\}$.
\end{remark}

We speculate that the on-diagonal lower bound, \eqref{eq:lbd} with $y=x_0$, holds in great generality, that is, for any growing $\G_t\uparrow\G_\infty$ of uniformly bounded degrees (though it may not be sharp); the isoperimetry of $\{\G_t\}$ and $\G_\infty$ will play a role in sharp lower estimates (including off-diagonal). Unfortunately, the cases we can handle now assume much more structures. It would be of much interest to bypass the isoperimetry \eqref{eq:regularity}, so as to resolve e.g. \cite[Conjecture 1.2]{DHS}.

\medskip
We deal with below $\{X_t\}$ on another class of slowly increasing $\{\bG_t\}$, comparable to ``simple" sets, where we can conclude (weak) recurrence without establishing lower heat kernel estimates.
The setting is again specialized, in that the subgraph is frozen for a period of time then changes to a much larger subgraph, but here we do not impose regularity assumptions on $\{\bG_t\}$ as done in (\ref{eq:regularity}). It is closely related to (but does not resolve) \cite[Conj. 1.2]{DHS}, which conjectures the recurrence of {\emph{any}} growing domains $\D_t$ of $\Z^d$, $d\ge 3$, which satisfy $\B_{f(t)}\subseteq\D_t\subseteq\B_{Cf(t)}$, as soon as $\int_1^\infty f(t)^{-d}dt=\infty$, the opposite direction being already proven in \cite[Theorem 1.4(a)]{DHS}. 

To preceed, let $\{t_l\}_{l\in\N}$ be an increasing sequence of integers, and $\{\K_l,\K^l\}_{l\in\N}$ be two sequences of increasing connected finite subgraphs of $\bG_\infty$. Assume $\{\bG_t\}$ satisfy the following three properties:

\begin{assumption}\label{sep-of-scale}
(a) $\{\K_l,\K^l\}$ are nested: for all $l\in\N$, $\K^{l-1}\subseteq\K_l\subseteq\K^l$. Their edge sets are {\emph{induced}} from $\G_\infty$, in the sense of Definition \ref{def:induce}. For $t_l\le t<t_{l+1}$, the graph is frozen:
\begin{align}\label{eq:sand}
\K_l\subseteq \bG_t=\bG_{t_l} \subseteq \K^l.
\end{align}
(b) For any $x,y\in\K^{l-1}$, the hitting distributions on $\partial\K_l$ of \abbr{SRW}-s starting from $x,y$ have uniformly bounded in $l$ Radon-Nikodyn density: for some constant $c_{\text{RN}}>0$ and all $x,y\in\K^{l-1}$, $z\in\partial\K_l$, $l\in\N$,
\begin{align}\label{eq:exit-dist}
c_{\text{RN}}\le \frac{\P_x(X_{\tau_{\K_l}}=z)}{\P_y(X_{\tau_{\K_l}}=z)}\le c_{\text{RN}}^{-1}.
\end{align}
The volume of $\K_l$ is of exponential growth in $l$:
\begin{align}\label{eq:growth}
\liminf_{l\to\infty}\{l^{-1}\log v(\K_l)\}>0.
\end{align}
(c) The exit time from $\K_l$ has light tails: for some positive $\ep$ and $c_E=c_E(\ep)$, all $l\in\N$, $s\ge 1$,
\begin{align}\label{eq:exit-tail}
\sup_{x\in\K_l}\bP_x\Big(\tau_{\K_l}>\frac{sv(\K_l)}{\big(\log v(\K_l)\big)^{2+\ep}}\Big)<c_E^{-1}e^{-c_Es}.
\end{align}
\end{assumption}

\begin{remark}
A main example that satisfies Assumption \ref{sep-of-scale} is $\bG_\infty=\Z^d$, $d>2$, with the ``simple" sets lattice balls in $\Z^d$: 
\begin{align}\label{simple-sets}
l\ge 1,\quad \K_l=\B(x_0,r_l), \quad \K^l=\B\big(x_0,r'_l\big), \quad r_{l+1}\ge(1+\delta)r'_l
\end{align}
for some fixed $\delta>0$. See \cite[Lemma 6.3.7]{LL} for (\ref{eq:exit-dist}), and \cite[Corollary 6.9.6]{LL} for (\ref{eq:exit-tail}).
 
In fact, the assumption (c) above can be guaranteed by good isoperimetry, such as if $\G_\infty$ satisfies the Faber-Krahn inequality of order $\theta>2$. That is,
for some $c_{\text{FK}}>0$, and all finite $A\subseteq V$,
\begin{align}\label{fk-ineq}
\lambda_1(A)\ge c_{\text{FK}}v(A)^{-2/\theta}, 
\end{align}
where $\lambda_1(A)$ denotes the smallest (Dirichlet) eigenvalue of the Laplace operator $L:=I-P$ on $\G_\infty$ with zero boundary condition on $V\backslash A$. The tail of $\tau_A$ is controlled by $\lambda_1(A)$, cf. \cite[section 6.9]{LL}. 

For assumption (b) it is sufficient to have a (scale invariant) elliptic Harnack inequality (\abbr{EHI}) on $\G_\infty$,  and as ``simple" sets the $\G_\infty$-balls of \eqref{simple-sets}. See \cite[Theorem 5.11]{BM} for a stable characterization of the \abbr{EHI}. In particular, the same HKE($\psi$) \eqref{sub-gauss-1}-\eqref{eq:poly-growth} would suffice. 
\end{remark}

\begin{ppn}\label{ppn:rec}
In the setting of Assumption \ref{sep-of-scale}, and further assume $\pi^{(t)}(x,x)\ge 1$ and $\pi^{(t)}(x)$ are uniformly bounded by constant $\Delta$. 
Then 
\begin{align}\label{eq:div}
\sum_{t=0}^\infty\frac{1}{v(t)}=\sum_{l=0}^\infty\frac{t_{l+1}-t_l}{v(t_l)}=\infty
\end{align}
is a sufficient criterion for $\{X_t\}$ to be (weak) recurrent. 
\end{ppn}

\section{Proof of Theorem \ref{thm-main} and Proposition \ref{ppn-main}}\label{sec:2}
We use the method of evolving random sets (\cite{MP}) to obtain the upper bounds of Theorem \ref{thm-main}. We recall its definition in time-dependent case from \cite[Definition 1.12]{DHMP}, {\emph{valid for $t\mapsto\pi^{(t)}(x)$ non-decreasing}} of a time-varying graph $\G_t$. Starting with $S_0=\{x_0\}\subseteq V_0$, inductively construct $S_{t+1}\subseteq V_{t+1}$ from $S_t\subseteq V_t$ and an independent uniform $[0,1]$ random variable $U_{t+1}$, as follows
\begin{align}\label{rule}
S_{t+1}:=\Big\{y\in V:\, \frac{\pi^{(t)}(S_t,y)}{\pi^{(t+1)}(y)}>U_{t+1}\Big\}.
\end{align}
In particular, given $S_t$ the probability that a vertex $y$ is in $S_{t+1}$ is given by
\begin{align*}
\bP(y\in S_{t+1}|S_t)=\frac{\pi^{(t)}(S_t,y)}{\pi^{(t+1)}(y)}\in[0,1].
\end{align*}
It is clear that $|S_t|<\infty$ for all $t\in\N$ with probability one.
The following key lemma relating the evolving sets $\{S_t\}$ to the walk $\{X_t\}$ is proved in \cite[Lemma 2.1]{DHMP}.
\begin{lem}\label{evol-set-lem}
The sequence $\{\pi^{(t)}(S_t)\}$ is a martingale and for any $t\in\N$, $x_0\in V_0, y\in V_t$, 
\begin{align}\label{eq:set-walk}
P(0,x_0;t,y)=\frac{\pi^{(t)}(y)}{\pi^{(0)}(x_0)}\bP_{\{x_0\}}(y\in S_t).
\end{align}
\end{lem}
Using the martingale property of $\{\pi^{(t)}(S_t)\}$ in Lemma \ref{evol-set-lem}, one defines another set-valued process called {\emph{size-biased}} evolving set $\{\wh{S}_t\}$ with $\wh{S}_0=\{x_0\}$, having Markovian transition probabilities
\begin{align}\label{size-bias-ker}
\wh{K}(t,A;t+1,B):=\frac{\pi^{(t+1)}(B)}{\pi^{(t)}(A)}K(t,A;t+1,B), \quad   A\subseteq V_t, B\subseteq V_{t+1},\, t\in\N,
\end{align}
where $K(t,\cdot;t+1,\cdot)$ is the inhomogeneous transition probabilities of the (original) evolving sets $\{S_t\}$. They induce the multi-step transition probabilities
\begin{align*}
\wh{K}(s,A;t,B)=\frac{\pi^{(t)}(B)}{\pi^{(s)}(A)}K(s,A;t,B), \quad   A\subseteq V_s, B\subseteq V_t,\, 0\le s<t,
\end{align*}
We use $\wh{\bE}$ to denote expectation under the $\{\wh{S}_t\}$ law.

\medskip
\noindent
{\emph{Proof of Theorem \ref{thm-main}. }} 
By (\ref{eq:set-walk}) we first bound $\bP_{\{x_0\}}(y\in S_t)$. Henceforth fixing $t\ge 2$ and any $1\le s\le t-1$, define for every $u\in\N$,
\begin{align*}
A_u=\Big\{\sup_{0\le i\le u}\pi^{(i)}(S_i)\le v(s)/2\Big\}.
\end{align*}
By Doob's martingale inequality,
\begin{align*}
\bP_{\{x_0\}}(A_t^c)\le \frac{2\pi^{(0)}(x_0)}{v(s)}.
\end{align*}
Thus, for any $\alpha\in (0,1)$,
\begin{align}\label{eq:moment}
\bP_{\{x_0\}}(y\in S_t)&\le \bP_{\{x_0\}}(A_t^c)+\bP_{\{x_0\}}(\{y\in S_t\}\cap A_t)\nonumber\\
&\le \frac{2\pi^{(0)}(x_0)}{v(s)}+\frac{1}{\pi^{(t)}(y)^\alpha}\bE_{\{x_0\}}\left[\pi^{(t)}(S_t)^\alpha 1_{A_t}\right].
\end{align}
By the same derivation as in \cite[Lemma 2.2]{DHMP} of \cite[Eq. (2.13)]{DHMP}, for $S_u\neq\emptyset$,
\begin{align}\label{eq:recur}
\bE[\pi^{(u+1)}(S_{u+1})^\alpha\big| \cF_u]\le \pi^{(u)}(S_u)^\alpha\big[1-c_+R_u^2\big]
\end{align}
where
\begin{align*}
R_u=\frac{\pi^{(u)}(S_u,S_u^c)}{\pi^{(u)}(S_u)}, \quad c_+=\frac{2\alpha(1-\alpha)\gamma^2}{(1-\gamma)^2}, \quad \cF_u=\sigma\{S_0,S_1,...,S_u\}.
\end{align*}
While we do not repeat this part of derivations, the key steps consist of computing one-step evolution of $\pi^{(u)}(S_u)$ conditional on $\{U_{u+1}\le 1/2\}$ or $\{U_{u+1}>1/2\}$ resp., by the update rule (\ref{rule}). Then use conditional Jenson's inequality to bound the \abbr{LHS} of (\ref{eq:recur}).

Since $A_u\in\cF_u$, $A_{u+1}\subseteq A_u$, one has from (\ref{eq:recur})
\begin{align}
\bE[\pi^{(u+1)}(S_{u+1})^\alpha 1_{A_{u+1}}\big| \cF_u]&\le\bE[\pi^{(u+1)}(S_{u+1})^\alpha 1_{A_u}\big| \cF_u]\nonumber\\
&\le\pi^{(u)}(S_u)^\alpha1_{A_u}(1-c_+R_u^2).\label{sub-mg-est}
\end{align}
On $1_{A_u}$, $u\ge s$, one has $\pi^{(u)}(S_u)\le v(s)/2\le v(u)/2$, hence
by the definition (\ref{def:isop}) of the isoperimetric function, 
\begin{align*}
R_u\ge\phi_u(\pi^{(u)}(S_u)),
\end{align*}
the preceding inequality \eqref{sub-mg-est} then yields for $u\ge s$
\begin{align}
\bE[\pi^{(u+1)}(S_{u+1})^\alpha 1_{A_{u+1}}\big| \cF_u]\le \pi^{(u)}(S_u)^\alpha1_{A_u}
\big[1-c_+\phi_u^2(\pi^{(u)}(S_u))\big].\label{eq:indicator}
\end{align}
Recall that $\wh{\bE}$ denotes the expectation over the law of size-biased evolving sets. From (\ref{size-bias-ker}) we see that $S_t\neq\emptyset$ for all $t$ with probability one under $\wh{\bE}$, starting from $S_0$ non-empty.
Let for $u\ge s$
\begin{align*}
Z_u:=\pi^{(u)}(S_u)^{\alpha-1}1_{A_u}, \quad 
L_u:=\wh{\bE}_{\{x_0\}}(Z_u)=\frac{\bE_{\{x_0\}}\big[\pi^{(u)}(S_u)^\alpha 1_{A_u}\big]}{\pi^{(0)}(x_0)}.
\end{align*}
The inequality (\ref{eq:indicator}) together with \eqref{size-bias-ker} yield for $u\ge s$
\begin{align}
\wh{\bE}\left[Z_{u+1}|\cF_u\right]\le Z_u\Big[1-c_+\phi_u^2\Big(Z_u^{\frac{1}{\alpha-1}}\Big)\Big],\label{eq:recur-2}
\end{align}
noting that in case $1_{A_u}=0$, necessarily also $1_{A_{u+1}}=0$ and the inequality also holds.
Set for $u\ge s$ the non-decreasing (in $z$) and non-negative
\begin{align*}
f_u(z):=\frac{c_+}{2}\phi_u^2\Big(\Big(\frac{z}{2}\Big)^{\frac{1}{\alpha-1}}\Big), \quad  z\ge 0,
\end{align*}
since $r\mapsto\phi_u(r)$ is non-increasing.
Use \cite[Lemma 12]{MP} on the \abbr{RHS} of (\ref{eq:recur-2}) after expectation, per $u\ge s$, we get
\begin{align}\label{L-iter}
L_{u+1}\le L_u[1-f_u(L_u)].
\end{align}
In particular, $u\mapsto L_u$ is non-increasing, and bounding \eqref{L-iter} further using $1-x\le e^{-x}, \forall x\ge 0$, we get
\begin{align}\label{eq:solve}
\int_{L_{u+1}}^{L_u}\frac{dz}{zf(z)}\ge\frac{1}{f_u(L_u)}\log\frac{L_u}{L_{u+1}}\ge 1,  \quad u\ge s.
\end{align}
By Jenson's inequality and the martingale property, 
\begin{align}\label{eq:Ls}
L_s= \frac{\E_{\{x_0\}}[\pi^{(s)}(S_s)^\alpha]}{\pi^{(0)}(x_0)}\le \frac{\big(\E_{\{x_0\}}[\pi^{(s)}(S_s)]\big)^\alpha}{\pi^{(0)}(x_0)}=\pi^{(0)}(x_0)^{\alpha-1}\le 1.
\end{align} 
We iteratively solve for $L_u, u=s+1,...,t$ out of (\ref{eq:solve}) starting from (\ref{eq:Ls}), and after a change of variables $z'=z/2$ the formula agrees with (\ref{eq:iter}). Combined with (\ref{eq:moment}) and (\ref{eq:set-walk}) it yields
\begin{align*}
P(0,x_0;t,y)\le\frac{\pi^{(t)}(y)}{\pi^{(0)}(x_0)}\Big[\frac{2\pi^{(0)}(x_0)}{v(s)}+\frac{\pi^{(0)}(x_0)L_t}{\pi^{(t)}(y)^\alpha}\Big]=\frac{2\pi^{(t)}(y)}{v(s)}+\pi^{(t)}(y)^{1-\alpha}L_t.
\end{align*}
Since we are free to choose $s\in[1,t-1]$ in the beginning, we get the first bound (\ref{first-bd}). 

Turning to show the second bound (\ref{second-bd}), note that $\phi_t(r)\ge\Phi_t$ for all $t\in\N$ and any $r\ge0$, hence (\ref{eq:iter}) can be further simplified to
\begin{align*}
L_{u+1}:=\argmax\Big\{\ell: \, \frac{1}{c_+\Phi^2_u}\log\frac{L_u}{\ell}\ge 1\Big\}, \quad  s\le u\le t-1.
\end{align*}
That is, upon iteration,
\begin{align*}
L_t\le L_se^{-c_+\sum_{u=s}^{t-1}\Phi_u^2}\le 
e^{-c_+\sum_{u=s}^{t-1}\Phi_u^2}.
\end{align*}
Taking $s=\lfloor t/2\rfloor$, $\alpha=1/2$, we get (\ref{second-bd}) upon noting $\pi^{(t)}(y)\in[1,\Delta]$.
\qed

\medskip
\noindent
{\emph{Proof of Proposition \ref{ppn-main}. }} Combining \eqref{second-bd}, (\ref{eq:regularity}) and the doubling property of $t\mapsto v(t)$, we have for some $C_1=C_1(\gamma,\Delta,D,\zeta)$ finite and all $t\ge 2$, $y\in V_t$,
\begin{align}\label{eq:diag-ubd}
P(0,x_0;t,y)\le \frac{C_1}{v(t)}.
\end{align}
In view of \eqref{cond:fluct}, \eqref{ball-like} and (\ref{eq:diag-ubd}), there exist some $\delta_0=\delta_0(C_1)\in(0,\frac{1}{2}]$ and $t_0=t_0(\delta_0)$ finite, such that for all $t\ge t_0$,
\begin{align*}
\sum_{y\in \G_t\backslash\B(x_0,(1-\delta_0)r(t))}P(0,x_0;t,y)\le 1/2.
\end{align*}
Consequently, for such $\delta_0$ and all $t\ge t_0$,
\begin{align*}
\bP_{x_0}(X_t\in\B(x_0,(1-\delta_0)r(t)))
=\sum_{y\in \B(x_0,(1-\delta)r(t))}P(0,x_0;t,y)\ge 1/2.
\end{align*}
By (\ref{ball-like}), (\ref{eq:dhk}), there exists $c_2=c_2(\delta_0,c_{\text{HK}})>0$ such that for all $t\ge t_0$, $y\in\B(x_0,(1-\delta_0)r(t))$,
\begin{align*}
\bP_{x_0}&(X_{t+\psi(r(t))}=y)\\
&\ge \bP\left(X_{t+\psi(r(t))}=y|X_t\in\B(x_0,(1-\delta_0)r(t))\right)\cdot
\bP_{x_0}(X_t\in\B(x_0,(1-\delta_0)r(t)))\\
&\ge\inf_{x\in\B(x_0,(1-\delta_0)r(t))}\left\{p_{\psi(r(t))}^{r(t)}(x,y)\right\}\cdot\frac{1}{2}\\
&\ge\frac{c_2}{v(\B(x_0,r(t)))}\ge \frac{c_2}{v(t)}.
\end{align*}
Applying a change of variables $t'=t+\psi(r(t))$, we get \eqref{eq:lbd} since $v(t')\ge v(t)$. Adjusting $c_2$ if necessary, the bound applies to all $t\ge 2$.
\qed

\section{Proof of Proposition \ref{ppn:rec}}\label{sec:3}
During each interval $I_l:=[t_l,t_{l+1})$, $\G_t=\G_{t_l}$ is a fixed finite graph. We first recall an on-diagonal lower bound that holds for all time.

\begin{lem}\label{lem:dlbd}
Let $\bH$ be a connected finite graph, and $\{Y_t\}$ a discrete time \abbr{SRW}. Then, for all $t\in\N$ and $x\in \bH$ we have that
\begin{align*}
\P_x(Y_{2t}=x)\ge\frac{\pi(x)}{\pi(\bH)}.
\end{align*}
\end{lem}
\begin{proof}
Using the same notation as \eqref{def:degree}-\eqref{def:pi} for $\pi$ replacing $\pi^{(t)}$, we have by reversibility
\begin{align*}
\frac{\P_x(Y_{2t}=x)}{\pi(x)}&=\sum_z\frac{\P_x(Y_t=z)}{\pi(z)}\frac{\P_z(Y_t=x)}{\pi(x)}\pi(z)=\sum_z\Big(\frac{\P_x(Y_t=z)}{\pi(z)}\Big)^2\pi(z)\\
&\ge \frac{\left[\sum_z\P_x(Y_t=z))\right]^2}{\pi(\bH)}=\frac{1}{\pi(\bH)}.
\end{align*}
\end{proof}
\noindent
{\emph{Proof of Proposition \ref{ppn:rec}. }}
First we make the (spurious) assumption that 
\begin{align}\label{spurious}
t_{l+1}-t_l\ge \frac{4Mv(\K_l)}{(\log v(\K_l))^{1+\ep/2}}
\end{align}
for some finite $M$, the $\ep>0$ of (\ref{eq:exit-tail}) and all $l\in\N$.

Henceforth fix $l\in\N$. By Lemma \ref{lem:dlbd}, for all $x\in V_{t_l}$, $t\in[t_l,t_{l+1})$,
\begin{align*}
\bP(X_t=x|X_{t_l}=x)\ge\frac{\pi(x)}{v(t_l)\Delta}.
\end{align*}
Then, setting $n=n(l):=\frac{t_{l+1}-t_l}{4}$,
\begin{align}\label{eq:n-2n}
\bE\left(|\{s\in[t_l+n,t_l+2n):\, X_s=x\}|\big|X_{t_l}=x\right)\ge
\frac{n\pi(x)}{v(t_l)\Delta}.
\end{align}
By (\ref{eq:growth})-(\ref{eq:exit-tail}) and our choice of $n$, uniformly for all $x\in \K^{l-1}\subseteq \K_l$,  
\begin{align}\label{eq:tau}
\bP(\tau_{\K_l}\ge n|X_{t_l}=x)  \text{  decays super-polynomially in }l.
\end{align}
Hence, for $x\in\K^{l-1}$,
\begin{align}\label{eq:0-n}
&\bE\left(|\{s\in [t_l, t_l+\tau_{\K_l}):\, X_s=x\}|\big|X_{t_l}=x\right)\nonumber\\
&\le\bE\left(|\{s\in [t_l, t_l+n):\, X_s=x\}|\big|X_{t_l}=x\right)+\bE(\tau_{\K_l}1_{\{\tau_{\K_l}\ge n\}}|X_{t_l}=x),
\end{align}
where further $\bE(\tau_{\K_l}1_{\{\tau_{\K_l}\ge n\}}|X_{t_l}=x)\le s(l)$ uniformly in $x$ for some summable $s(l)$ due to the tail bound (\ref{eq:tau}).
Combining (\ref{eq:n-2n}) and (\ref{eq:0-n}), 
\begin{align}\label{eq:0-to-2n}
&\bE\left(|\{s\in [t_l+\tau_{\K_l}, t_l+\tau_{\K_l}+2n]:\, X_s=x\}|\big|X_{t_l}=x\right)\nonumber\\
&\ge\bE\left(|\{s\in [t_l, t_l+2n]:\, X_s=x\}|\big|X_{t_l}=x\right)-\bE\left(|\{s\in [t_l, t_l+\tau_{\K_l}]:\, X_s=x\}|\big|X_{t_l}=x\right)\nonumber\\
&\ge\bE\left(|\{s\in[t_l+n,t_l+2n):\, X_s=x\}|\big|X_{t_l}=x\right)-s(l)\ge\frac{n\pi(x)}{v(t_l)\Delta}-s(l).
\end{align}
By (\ref{eq:exit-dist}), the first hitting distribution of $\partial\K_l$ from $x_0,x\in\K^{l-1}$ have uniformly bounded Radon-Nikodyn density. Employing strong Markov property at $\tau_{\K_l}$, this extends to any measurable function of the hitting distribution. Hence, combining \eqref{eq:exit-dist}) and \eqref{eq:0-to-2n} we have that
\begin{align}\label{change-pt}
&\bE\left(|\{s\in [t_l+\tau_{\K_l}, t_l+\tau_{\K_l}+2n]:\, X_s=x\}|\big|X_{t_l}=x_0\right)\nonumber\\
&\ge c_{RN}\bE\left(|\{s\in [t_l+\tau_{\K_l}, t_l+\tau_{\K_l}+2n]:\, X_s=x\}|\big|X_{t_l}=x\right)\ge c_{RN}\left(\frac{n\pi(x)}{v(t_l)\Delta}-s(l)\right).
\end{align}
Employing again \eqref{eq:tau}, we deduce from \eqref{change-pt}
\begin{align}
&\E\left(|\{s\in[t_l,t_l+4n):\, X_s=x\}|\big|X_{t_l}=x_0\right)\nonumber\\
&\ge\E\left(|\{s\in[t_l,t_l+\tau_{\K_l}+2n):\, X_s=x\}|\big|X_{t_l}=x_0\right)-\E(2\tau_{\K_l}1_{\{\tau_{\K_l}\ge 2n\}}|X_{t_l}=x_0)\nonumber\\
&\ge c_{RN}\left(\frac{n\pi(x)}{v(t_l)\Delta}-s(l)\right)-s'(l),\label{lbd:x0}
\end{align}
where $s'(l)$ is another summable term in $l$.

Recall that $t_{l+1}-t_l= 4n$, and $\G_t=\G_{t_l}$ during $I_l$. We use reversibility to deduce from \eqref{lbd:x0} that
\begin{align}\label{eq:loc-time}
&\bE\left(|\{s\in [t_l, t_{l+1}):\, X_s=x_0\}|\big|X_{t_l}=x\right)\nonumber\\
&=\frac{\pi(x_0)}{\pi(x)}\bE\left(|\{s\in [t_l, t_{l+1}):\, X_s=x\}|\big|X_{t_l}=x_0\}\right)\nonumber\\
&\ge \frac{c_{RN}}{\Delta}\left(\frac{(t_{l+1}-t_l)\pi(x)}{4v(t_l)\Delta}-s(l)\right)-\frac{s'(l)}{\Delta}.
\end{align}
Summing over $l\in\N$ establishes the recurrence of $\{X_t\}$ as soon as (\ref{eq:div}) holds, subject to \eqref{spurious}.

We now remove the spurious assumption \eqref{spurious}. By (\ref{eq:growth}),
\begin{align*}
\sum_{l=0}^\infty\frac{v(\K_l)/(\log v(\K_l))^{1+\ep/2}}{v(t_l)}
\le\sum_{l=0}^\infty(\log v(\K_l))^{-1-\ep/2}<\infty,
\end{align*}
hence if (\ref{eq:div}) holds, there must exist a subsequence $\{l_k\}\subseteq\N$ such that 
\begin{align*}
\forall k,\quad t_{l_k+1}-t_{l_k}\ge \frac{4Mv(\K_{l_k})}{(\log v(\K_{l_k}))^{1+\ep/2}} \quad \&\quad \sum_{k=0}^\infty\frac{t_{l_k+1}-t_{l_k}}{v(t_{l_k})}=\infty.
\end{align*}
Consequently, we merely restrict the whole analysis to the intervals $\{I_{l_k}: k\in\N\}$.
\qed

\section{Proof of Proposition \ref{ppn:no-merging}}\label{sec:4}
For each $t\in\N$ assigning $\Pi_t:=\{\pi^{(t)}(x,y);\, |x-y|\le1\}$ on the edges of $V_N=\{0,...,N\}$, induces a reversible pair of transition kernel $K^{(t)}(x,y):=\pi^{(t)}(x,y)/\pi^{(t)}(x)$ and probability measure $\mu^{(t)}(x):=\pi^{(t)}(x)/\sum_{y\in V_N}\pi^{(t)}(y)$. We mention in passing that the notation $K_{0,t}(x,y)$ used in \cite{SZ1,SZ} and adopted in Proposition \ref{ppn:no-merging} can also be written as 
\begin{align*}
K_{0,t}(x,y)=\P(X_t=y|X_0=x)=P(0,x;t,y).
\end{align*}

The following (counter-)example can be first described in words: the time-dependent $\Pi_t$ we assign to $V_N$ produces drift that points towards $0$ on the space interval $[0,N/2]$, and points towards $N$ on the space interval $[N/2, N]$. Hence, two Markov chains starting from vertices $0$ and $N$ respectively cannot couple in order $N^2$ time.

\medskip
\noindent
{\emph{Proof of Prop. \ref{ppn:no-merging}. }}
Without loss of generality assume $N$ is even.
For any $\theta,\eta>0$, assign the time-dependent $\Pi_t$ to $V_N$ as
\begin{align*}
\text{For }&1\le x\le N/2,\\
&\pi^{(t)}(x-1,x)=1+\theta, \, \pi^{(t)}(x,x)=1-\eta, \, \text {when }x+t\text{ is even};\\
&\pi^{(t)}(x-1,x)=1-\theta, \, \pi^{(t)}(x,x)=1, \, \text {when }x+t\text{ is odd}.\\
\text{For }&N-1\ge  x\ge N/2,\\
&\pi^{(t)}(x,x+1)=1+\theta, \, \pi^{(t)}(x,x)=1-\eta, \, \text {when }x+t\text{ is even};\\
&\pi^{(t)}(x,x+1)=1-\theta, \, \pi^{(t)}(x,x)=1, \, \text {when }x+t\text{ is odd};\\
\text{For }&x=0 \text{ or }N,\, 
\pi^{(t)}(x,x)=2.
\end{align*}
It is easy to check that the requirements on $(K^{(t)},\pi^{(t)})$ are met, upon taking $\theta,\eta$ small enough relative to the given $\ep>0$.
Consider a Markov chain $\{X_t\}$ starting at vertex $0$ at time $0$. We define an auxiliary, two-state homogeneous Markov chain, $\{Z_t\}$, as
\begin{align*}
Z_t&=A, \quad \text{if }\pi^{(t)}(X_t,X_t+1)=1-\theta,\\
Z_t&=B, \quad \text{if }\pi^{(t)}(X_t,X_t+1)=1+\theta.
\end{align*}
By construction, a lazy step for $\{X_t\}$ is equivalent to a change of state for $\{Z_t\}$, with the latter having the following transition matrix as long as $X_t$ stays in the interval $[1,N/2)$
\begin{align*}
\bordermatrix{~ & A & B \cr
                  A & \frac{2}{3-\eta} & \frac{1-\eta}{3-\eta} \cr
                  B & \frac{1}{3} & \frac{2}{3} \cr}.
\end{align*}
It converges exponentially fast to its invariant distribution 
\begin{align}\label{Z:inv}
[u(A),u(B)]:=\left[\frac{3-\eta}{6-4\eta},\frac{3-3\eta}{6-4\eta}\right]. 
\end{align}
Furthermore, when $Z_t=A$, $X_t$ has drift $\Delta(A)=-2\theta/(3-\eta)$ towards $0$ for its next transition; when $Z_t=B$, $X_t$ has drift $\Delta(B)=2\theta/3$ towards $N/2$ for its next transition, with
\begin{align}\label{drift}
\beta=\beta(\theta,\eta):=u(A)\Delta(A)+u(B)\Delta(B)<0.
\end{align}
Define $\sigma_0=\sigma^N_0:=\inf\{t>0:\, X_t=0\}$, then for some $c_1=c_1(\beta)$ positive and all $N$,
\begin{align}\label{eq:excur}
\P(\sigma_0\ge N/2|X_0=0)\le c_1^{-1}e^{-c_1N}.
\end{align}
Indeed, we construct $\{X_t\}$ and some other symmeric lazy random walk $\{Y_t\}$ on $\Z$ with $Y_0=0$ on the same probability space $(\Omega, \cF,\wt{\P})$, as follows. At each step $t\in\N$, take $U^{(t)}$ an independent Uniform $[0,1]$ coin. On the event $Z_t=A$, let $X$ and $Y$ both stay put if the outcome $U^{(t)}\le(1-\eta)/(3-\eta)$, or else both move as symmetric \abbr{SRW}s if $U^{(t)}\ge 1-(2-2\theta)/(3-\eta)$, or else let $X$ move left and $Y$ move as symmetric \abbr{SRW}. Similarly, on the event $Z_t=B$, let $X$ and $Y$ both stay put if $U^{(t)}\le 1/3$, or else both move as symmetric \abbr{SRW}s if $U^{(t)}\ge (2-2\theta)/3$, or else let $X$ move right and $Y$ move as symmetric \abbr{SRW}. 

Though its laziness depends on $Z_t$, the process $\{Y_t\}$ is balanced and clearly diffusive. Furthermore, we have shown that $\{Z_t\}$ is ergodic with explicit invariant measure \eqref{Z:inv} and exponentially fast convergence, implying that under the coupling
\begin{align*}
\forall t, \quad \wt{\P}(E_t|X_0=Y_0=0)\ge 1- c^{-1}e^{-ct}, \quad \text{for }E_t:=\{X_t-Y_t\le \beta t/2\}.
\end{align*}
To confirm \eqref{eq:excur}, 
\begin{align*}
&\wt{\P}(\sigma_0\ge N/2|X_0=0)=\wt{\P}(X_t\in[1,N/2], \, \forall t\in[1,N/2]|X_0=0)\\
&\le \wt{\P}(\{X_{N/2}\in[1,N/2]\}\cap E_{N/2}|X_0=Y_0=0)+\wt{\P}(E_{N/2}^c|X_0=Y_0=0)\\
&\le \wt{\P}(Y_{N/2}\ge 1-\beta N/2|X_0=Y_0=0)+\wt{\P}(E_{N/2}^c|X_0=Y_0=0)\\
&\le c_1^{-1}e^{-c_1N},
\end{align*}
the last inequality being due to the diffusivity of $Y$.
The same deviation bound as (\ref{eq:excur}) holds also for $\sigma_k-\sigma_{k-1}$ replacing $\sigma_0$, where $\sigma_k=\sigma^N_k:=\inf\{t>\sigma_{k-1}: \, X_t=0\}$, $k\ge1$. 
By the strong Markov property, the mutual independence of the increments $\left\{(X_t)_{t\in[\sigma_{k-1}, \sigma_k)}\right\}_{k\ge1}$ implies that with $\alpha:=c_1/2$ and $T_N:=e^{\alpha N}$, 
\begin{align}
&\P(X_t< N/2, \, \forall t\in[0,T_N]|X_0=0)\ge \P\left(\cap_{k=1}^{T_N}\{\sigma_k-\sigma_{k-1}< N/2\}|X_0=0\right)\nonumber\\
&\ge \left(1-c_1^{-1}e^{-c_1 N}\right)^{T_N}\ge 1-c_1^{-1}T_Ne^{-c_1N}
=1-c_1^{-1}e^{-\alpha N}.
\label{eq:ld-x}
\end{align}
Similarly, one then shows that if $\{X_t\}$ starts from vertex $N$ at time $0$, also
\begin{align}\label{eq:ld-y}
\P(X_t>N/2, \, \forall t\in[0,T_N]|X_0=N)\ge 1-c_1^{-1}e^{-c_1N}.
\end{align}
To conclude, just note that $|K_{0,T_N}(0,A)-K_{0,T_N}(N,A)|\ge 1/2$ for $A:=[0,N/2]$ by \eqref{eq:ld-x}-\eqref{eq:ld-y}, hence $T_{TV}(1/2)\ge T_N$. Also, by (\ref{eq:ld-x}) there exists some $z\in[0,N/2]$ with $K_{0,T_N}(0,z)\ge 1/N$, whereas by \eqref{eq:ld-y} $K_{0,T_N}(N,z)\le c_2^{-1}e^{-c_2N}$, hence $T_\infty(1/2)\ge T_N$ as well.
\qed

\bigskip
\noindent
{\bf{Acknowledgments.}} I thank my advisor A. Dembo for many inspiring discussions (particularly on the role of mixing) and very helpful comments. I also thank Prof. T. Kumagai for introducing me to heat kernel estimates in 2015, and further valuable comments.


\begin{thebibliography}{99}
\bibitem{ABGK}
G. Amir, I. Benjamini, O. Gurel-Gurevich and G. Kozma. Random walk in changing environment.

\bibitem{ACDS}
S. Andres, A. Chiarini, J.-D. Deuschel and M. Slowik. Quenched invariance principle for random walks with time-dependent ergodic degenerate weights. {\emph{Ann. Probab.}} {\bf{46}} (2018), 302-336.


\bibitem{BBK}
M.T. Barlow, R.F. Bass and T. Kumagai. Stability of parabolic Harnack inequalities on metric measure spaces. {\emph{J. Math. Soc. Japan}} {\bf{58}} (2006), 485-519.

\bibitem{BGK}
M.T. Barlow, A. Grigor'yan and T. Kumagai. On the equivalence of parabolic Harnack inequalities and heat kernel estimates. {\emph{J. Math. Soc. Japan}} {\bf{64}} (2012), 1091--1146. 

\bibitem{BM}
M.T. Barlow and M. Murugan. Stability of the elliptic Harnack inequality. {\it{Ann. of Math. (2)}} {\bf{187}} (2018), 777-823.


\bibitem{Del}
T. Delmotte. Parabolic Harnack inequality and estimates of Markov chains on graphs. {\emph{Rev. Mat. Iberoam.}} {\bf{11}} (1999), 181-232.

\bibitem{DHS}
A. Dembo, R. Huang and V. Sidoravicius. Walking within growing domains: recurrence versus transience. {\emph{Elect. J. Probab.}} {\bf 19} (2014), no. 106, 1-20.

\bibitem{DHMP}
A. Dembo, R. Huang, B. Morris and Y. Peres. Transience in growing subgraphs via evolving sets. {\emph{Ann. Inst. H. Poincar\'e Prob. Stat.}} {\bf 53} (2017), 1164-1180.

\bibitem{DHZ}
A. Dembo, R. Huang and T. Zheng. Random walks among time increasing conductances: heat kernel
estimates. 


\bibitem{GP}
G. Giacomin and G. Posta. On recurrent and transient sets of inhomogeneous symmetric random walks. {\emph{Elect. Commun. Probab.}} {\bf{6}} (2001), 39--53. 

\bibitem{HS}
W. Hebisch and L. Saloff-Coste. Gaussian estimates
for Markov chains and random walks on groups. {\emph{Ann. Probab.}}
\textbf{{21}} (1993), 673-709.

\bibitem{HHSST}
M. Hil\'ario, F. den Hollander, R.S. dos Santos, V. Sidoravicius and A. Teixeira. Random walk on random walks. {\emph{Elect. J. Probab.}} {\bf{20}} (2015), no. 95, 1-35.



\bibitem{LL}
G.F. Lawler and V. Limic. {\emph{Random walk: a modern introduction.}} Cambridge studies in advanced mathematics {\bf 123}. Cambridge University Press, Cambridge 2010.

\bibitem{LBG} 
G.F. Lawler, M. Bramson, and D. Griffeath. Internal diffusion limited aggregation. {\emph{Ann. Probab.}} {\bf{20}} (1992), 2117--2140.

\bibitem{LPW}
D.A. Levin, Y. Peres and E.L. Wilmer. {\emph{Markov chains and mixing times.}} Amer. Math. Soc. (2009).

\bibitem{MO}
J.-C. Mourrat and F. Otto. Anchored Nash inequalities and heat kernel bounds for static and dynamic degenerate environments. {\emph{J. Funct. Anal.}} {\bf{260}} (2016), 201-228. 

\bibitem{MP1}
B. Morris and Y. Peres. Evolving sets and mixing. {\emph{Proceedings of the thirty-fifth annual ACM symposium on Theory of computing}} (2003), 279-286, ACM.

\bibitem{MP}
B. Morris and Y. Peres. Evolving sets, mixing and heat kernel bounds. {\emph{Probab. Th. Rel. Fields}} {\bf{133}} (2005), 245-266.

\bibitem{SZ1}
L. Saloff-Coste and J. Z\'u\~niga. Merging for inhomogeneous
finite Markov chains, part I: singular values and stability. {\emph{Elect.
J. Probab.}} \textbf{{14}} (2009), 1456-1494.

\bibitem{SZ}
L. Saloff-Coste and J. Z\'u\~niga. Merging for inhomogeneous finite Markov chains, part II: Nash and log-Sobolev inequalities. {\emph{Ann. Probab.}} {\bf{39}} (2011), 1161-1203.

\end{thebibliography}
\end{document}